\documentclass[a4paper, 12pt]{amsart}
\usepackage[utf8]{inputenc}
\usepackage[margin=1in]{geometry} 
\usepackage{amsmath,amsthm,amssymb,amsfonts,amscd,tikz,tikz-cd,tikz-3dplot,bbm,nicematrix}
\usepackage{thmtools}
\usepackage{thm-restate}
\usepackage{mathtools}
\usepackage{mathrsfs}
\usepackage{indentfirst}
\usepackage{enumitem}
\usepackage{graphicx}
\usepackage{subcaption}
\usepackage{float}
\usepackage{array}
\usepackage{hyperref}
\setlist[enumerate,1]{label={(\alph*)}}
\usetikzlibrary{decorations.markings}

\title[Invertible sheaves on subvarieties of the isomeric supergrassmannian]{Invertible sheaves and $\Pi$-invertible sheaves on the isomeric supergrassmannian and its toric subvarieties}
\author{Eric Jankowski}
\date{}

\newcommand{\F}{\mathscr{F}}

\newcommand{\J}{\mathscr{J}}
\newcommand{\K}{\mathscr{K}}
\renewcommand{\L}{\mathscr{L}}

\renewcommand{\O}{\mathscr{O}}

\renewcommand{\AA}{\mathbb{A}}
\newcommand{\PP}{\mathbb{P}}

\newcommand{\ZZ}{\mathbb{Z}}

\newcommand{\RR}{\mathbb{R}}
\newcommand{\CC}{\mathbb{C}}

\renewcommand{\epsilon}{\varepsilon}

\newcommand{\q}{\mathfrak{q}}
\newcommand{\uq}{\mathfrak{uq}}

\renewcommand{\t}{\mathfrak{t}}

\newcommand{\red}{\text{red}}
\newcommand{\thick}{\text{thick}}
\newcommand{\0}{{\ol{0}}}
\newcommand{\1}{{\ol{1}}}

\DeclareMathOperator{\rk}{rk}

\DeclareMathOperator{\Gr}{Gr}
\DeclareMathOperator{\QGr}{QGr}
\DeclareMathOperator{\GL}{GL}

\DeclareMathOperator{\Stab}{Stab}

\DeclareMathOperator{\Span}{Span}

\DeclareMathOperator{\cl}{cl}
\DeclareMathOperator{\Pic}{Pic}

\newtheorem{theorem}{Theorem}
\newtheorem{lemma}{Lemma}[section]

\newtheorem{proposition}[lemma]{Proposition}

\theoremstyle{definition}
\newtheorem{example}[lemma]{Example}
\newtheorem{remark}[lemma]{Remark}
\newtheorem{definition}[lemma]{Definition}

\newcommand{\ol}[1]{\overline{#1}}

\begin{document}

\begin{abstract}
We provide an elementary proof that with the exceptions of certain $\Pi$-projective spaces, both the Picard group and the $\Pi$-Picard set of the isomeric (i.e.\ type-Q) supergrassmannian are trivial. We extend this technique to show that the Picard group and the $\Pi$-Picard set of a supertorus orbit closure within the isomeric supergrassmannian can be easily calculated from its defining polytope by counting the number of simplex factors. Since the presence of nontrivial invertible sheaves and $\Pi$-invertible sheaves depends entirely on factors of $\Pi$-projective space, we construct them as symmetric powers of the tautological sheaf and its dual.
\end{abstract}


\maketitle

\section{Introduction and some reference ideas}

The isomeric supergrassmannian, $\QGr(r,n)$, is a homogeneous space for the supergroup $Q(n)$, isomorphic to its quotient by a maximal parabolic subgroup. By considering the action of a Cartan subgroup $Q(1)^n \subseteq Q(n)$ on $\QGr(r,n)$, one finds that its orbit closures are normal, making them toric supervarieties as defined by the author in \cite{Jankowski2}. These toric supervarieties admit a description by subobjects of the superpolytope which occurs as the image of an ``odd momentum map" on $\QGr(r,n)$.

In this paper we study two different kinds of locally free sheaves on $\QGr(r,n)$ and its toric subvarieties: invertible sheaves and $\Pi$-invertible sheaves. An invertible sheaf is merely a $(1|0)$-dimensional vector bundle; as usual, these assemble into the Picard group $\Pic(X)$. A $\Pi$-invertible sheaf is a $(1|1)$-dimensional vector bundle equipped with an odd involution; these assemble into the $\Pi$-Picard set $\Pic_\Pi(X)$, which is not a group but a pointed set with distinguished point $\O_X \oplus \Pi \O_X$. While invertible sheaves characterize morphisms to projective superspace $\PP^{m|n}$, $\Pi$-invertible sheaves characterize morphisms to $\Pi$-projective superspace $\PP_\Pi^n \cong \QGr(1,n+1)$. Such embeddings have been studied by various authors, in e.g. \cite{LPW}, \cite{Kwok}, and \cite{PolishchukGrassmannian}.

The study of $\Pi$-invertible sheaves originates with Manin and Skornyakov in \cite{ManinGFTCG}, where it is expounded that $\Pic_\Pi(X) \cong H^1(X,\O_X^\times)$. In many ways this isomorphism makes the $\Pi$-Picard set a good analogue of the classical Picard group of a variety, but in other ways it is not so nice. For instance, although the tensor product of two $\Pi$-invertible sheaves decomposes into a sum of two rank-$(1|1)$ locally free sheaves (see e.g.\ \cite{MVP} or \cite{ManinTNCG}), the summands cannot in general be equipped with odd involutions to make them $\Pi$-invertible. 

Nonetheless, the calculation of $\Pic_\Pi(X)$ for various supervarieties $X$ remains an interesting problem. Cacciatori and Noja calculate it for projective superspace $\PP^{m|n}$ in \cite[Corollaries 1 and 2]{CN}. Manin and Skornyakov do so for $\Pi$-projective superspace $\PP_\Pi^n$ in \cite[Chapter 2, Section 8.11.f]{ManinTNCG}. Manin, Voronov, and Penkov generalize this to $\QGr(r,n)$ in \cite[p.\ 2078]{MVP}, although the description is somewhat opaque and does not clearly indicate when the result is 0.

There have also been some efforts to find the Picard groups of various supervarieties. The result for the type-A supergrassmannian is due to Penkov and Skornyakov in \cite{PenkSkorn}, while Cacciatori and Noja treat the exceptional case of $\PP^{1|n}$ in \cite{CN}. The computation of $\Pic(\QGr(r,n))$, which will be done in this paper, is known to experts and often attributed to Penkov \cite{Penkov}. However, neither he nor the author could locate a written copy of his proof.

The Picard superscheme captures some subtleties of superschemes over a base; Bruzzo, Hernández Ruipérez, and Polishcuk calculate in \cite{BRP} that the Picard superscheme of $\PP_\Pi^n$ is $\AA^{0|1}$, in contrast to the Picard group $\Pic(\PP_\Pi^n)=0$. While Picard superschemes are not the focus of this paper, determining them for the examples at hand would constitute an interesting, if rather straightforward, future direction.

On the classical side, Picard groups of toric varieties are well-studied and combinatorial. If $X$ is a toric variety whose fan $\Sigma$ contains at least one cone of maximal dimension, then $\Pic(X) \cong \ZZ^{r-n-\lambda}$ where $r$ is the number of rays in $\Sigma$, $n$ is its dimension, and $\lambda$ is the dimension of the space of linear dependencies of rays of maximal cones (see e.g.\ \cite[Theorem VII.2.16]{Ewald}).

Our contribution to this ongoing story is the following result, which is the amalgamation of Propositions \ref{prop:PicQGr}, \ref{prop:PicTSV}, \ref{prop:PiPicQGr}, and \ref{prop:PiPicTSV}.

\begin{theorem}
    Let $X$ be $\QGr(r,n)$ or the closure of a $Q(1)^n$ orbit inside $\QGr(r,n)$. Then
    \begin{align*}
        \Pic(X) &\cong \prod_{\text{factors } \PP_\Pi^{1} \text{of } X} \ZZ \quad\quad\quad\text{and} \\
        \Pic_\Pi(X) &\cong \left\{ (\ell_i, c_i) \in  \prod_{\text{factors } \PP_\Pi^{d_i} \text{ of } X} \ZZ \oplus \CC \; \Bigg| \; \text{at most one } c_i \neq 0, \text{and if } d_i>1, \text{then } c_i^2=\ell_i \right\}
    \end{align*}
\end{theorem}

The number of factors of $\Pi$-projective space $\PP_\Pi^{d} \cong \QGr(1, d+1)$ can be detected easily by counting the number of $d$-simplex factors in the defining polytope. We emphasize that for $X=\QGr(r,n)$, we have
\begin{align*}
    \Pic(\QGr(r,n)) &= \begin{cases}
        \ZZ & r=1, n=2 \\
        0 & \text{otherwise}
    \end{cases} \\
    \Pic_\Pi(\QGr(r,n)) &= \begin{cases}
        \ZZ \oplus \CC & r=1, n=2 \\
        \{(\ell, c) \in \ZZ \oplus \CC \mid c^2=\ell\} &r\in \{1, n-1\}, n>2 \\
        0 & \text{otherwise}
    \end{cases}
\end{align*}
so in particular $\Pic(\QGr(r,n)) = \Pic_\Pi(\QGr(r,n)) = 0$ unless $\QGr(r,n) \cong \PP_\Pi^{n-1}$.

The setting $\ZZ \oplus \CC$ of $\Pic(X)$ (for $X=\PP_\Pi^n$) arises as the Picard group of the (non-super) scheme with the same topological space and the structure sheaf $\O_X/\J_X^2$, where $\J_X$ is the odd ideal. Note that the $\ZZ/2\ZZ$-grading must be forgotten in order to make it into a non-super scheme, though it is equivalent to consider the $\Pi$-Picard set of the corresponding superscheme. Equivalently, $\ZZ$ is the Picard group of the underlying variety (given by $\O_X/\J_X$), and $\CC \cong H^1(X,\Omega^1)$.

The tautological sheaves $\O_\Pi(-1)$ and $\O_\Pi(1)$ of $\PP_\Pi^n$ have been studied in e.g.\ \cite{ManinTNCG}, \cite{PolishchukGrassmannian}, and \cite{Kwok}. We generalize their construction to the $\Pi$-invertible sheaves $\O_\Pi(\ell)$. Since the choice of square root $c = \pm \sqrt{\ell}$ determines only the odd involution of the $\Pi$-invertible sheaf, these $\O_\Pi(\ell)$ form a complete set of representatives of $\Pic_\Pi(\PP_\Pi^n)$ for $n>1$. The following theorem summarizes the observations of Remarks \ref{rem:PicTSV}, \ref{rem:PiPicQGr}, and \ref{rem:PiPicTSV}.

\begin{theorem}
    Let $X$ be the closure of a $Q(1)^n$ orbit inside $\QGr(r,n)$. Then:
    \begin{enumerate}
        \item An invertible sheaf on $X$ is given by the tensor product of an invertible sheaf $\O(\ell), \ell \in \ZZ$ on each factor of $\PP_\Pi^1$ in $X$.
        \item A $\Pi$-invertible sheaf on $X$ is given by the tensor product of such an invertible sheaf with either:
        \begin{enumerate}[label=(\roman*)]
            \item A $\Pi$-invertible sheaf $\O_\Pi(\ell)$, $\ell \in \ZZ$ on some factor $\PP_\Pi^d$ for $d>1$, or
            \item A $\Pi$-invertible sheaf arising from $H^1(\PP^1, \Omega^1) \cong \CC$ on some factor of $\PP_\Pi^1$.
        \end{enumerate}
    \end{enumerate}
\end{theorem}

\subsection{Structure of Paper} We establish some notation and preliminaries in section \ref{sec:Notation&Prelims} before calculating Picard groups in section \ref{sec:Pic} and $\Pi$-Picard sets in section \ref{sec:PiPic}.

\subsection{Acknowledgments}

The author is grateful to Vera Serganova and Ivan Penkov for many helpful discussions. This material is based upon work supported by the National Science Foundation Graduate Research Fellowship Program under Grant No.\ 2146752.

\section{Notation and Preliminaries}\label{sec:Notation&Prelims}
We work over $\CC$.

\subsection{Toric Supervarieties}

As in \cite{Jankowski2}, we use supervariety to mean a GFRR (i.e.\ generically fermionically regular and reduced) separated superscheme of finite type over an algebraically closed field. A normal supervariety is one that in addition satisfies the $(R_1)$ and $(S_2)$ criteria of Serre (subsuming the GFRR condition), and a toric supervariety is an irreducible normal supervariety equipped with an action by a supertorus for which there is a dense open orbit.

A supertorus is a supergroup such as $Q(1)^n \subseteq Q(n)$ that occurs as a Cartan subgroup of a quasireductive supergroup; equivalently, it is a supergroup whose even part is a central torus.

For the purposes of this paper, a toric supervariety is determined by the usual data of a polyhedral fan $\Sigma$, together with a decoration $V_\rho$ on each of its rays $\rho$. The decorations are 0- or 1-dimensional subspaces of the odd part of the Lie superalgebra $\q(1)^n$, such that whenever $\rho$ and $\rho'$ are rays in the same cone $\sigma \in \Sigma$, it holds that $[V_\rho, V_{\rho'}] \subseteq \Span_\CC \sigma$. For more on supertori and toric supervarieties, see the author's prior papers \cite{Jankowski} and \cite{Jankowski2}.

If $X=(|X|, \O_X)$ is a supervariety, we write $i_X : \O_X \to \K_X$ for the natural embedding of the sheaf of regular functions into the sheaf of rational functions. Then, writing $(\mathscr{K}_{X,\1})$ for the ideal sheaf generated by odd elements in $\K_X$, there is a filtration $$\O_X \supset i_X^{-1}((\K_{X,\1})) \supset i_X^{-1}((\K_{X,\1})^2) \supset ... \supset i_X^{-1}((\K_{X,\1})^n) \supset i_X^{-1}((\K_{X,\1})^{n+1}) = 0$$
where $n$ is the odd dimension of $X$. We write
$$\J_X = i_X^{-1}((\K_{X,\1}))$$
for the ideal sheaf of the underlying variety $X_\red$ in $X$, and
$$\F_X = \J_X / i_X^{-1}((\K_{X,\1})^2)$$
for the \textit{fermionic sheaf} of $X$. The fermionic sheaf is essentially an odd cotangent sheaf that smooths out pill-type singularities (defined in \cite[Example 3.4]{Jankowski2}). For the purposes of $\QGr(r,n)$ and its subvarieties of concern, we have $i_X^{-1}((\K_{X,\1})^j) = \J_X^j$, and so $\F_X = \J_X/\J_X^2$.

\subsection{The Isomeric Supergrassmannian}

We write $\QGr(r,n)$ for the isomeric supergrassmannian. As mentioned above, it is a homogeneous space for $Q(n)$, isomorphic to $Q(n)/P_{Q(n)}$ for a maximal parabolic $P_{Q(n)} \subseteq Q(n)$.

Another important viewpoint on $\QGr(r,n)$ is as a super moduli space of $(r|r)$-dimensional $\Pi_V$-invariant subspaces of an $(n|n)$-dimensional isomeric super vector space $(V, \Pi_V)$, per the following definition.
\begin{definition}
    An \textit{isomeric super vector space} is a pair $(V, \Pi_V)$ consisting of a super vector space $V = V_\0 \oplus V_\1$ equipped with an odd involution $\Pi_V : V \to V$.
\end{definition}
That is, if $A$ is a commutative superalgebra, the $A$-points of $\QGr(r,n)$ parameterize free, rank-$(r|r)$, $\Pi_{A \otimes V}$-invariant submodules of $A \otimes V$. In coordinates, this means $\QGr(r,n)(A)$ is the space of full-rank matrices of the form
\begin{align*}
M = \begin{pmatrix}
    a_{11} + \eta_{11} & \cdots & a_{1r}+\eta_{1r} \\
    \vdots & \ddots & \vdots \\
    a_{n1} + \eta_{n1} & \cdots & a_{nr} + \eta_{nr}
\end{pmatrix}
\end{align*}
for $a_{ij} \in A_\0$ and $\eta_{ij} \in A_\1$, modulo the right action of $Q(r)(A)$. As seen in \cite{Noja}, it is covered by $\binom{n}{r}$ affine charts (one for each size-$(n-r)$ subset $I$ of $\{1,...,n\}$), consisting of matrices of the form
\begin{align*}
M' = \begin{pmatrix}
    a_I + \eta_I \\
    1_{r \times r}
\end{pmatrix}
\end{align*}
where $a_I + \eta_I$ is the submatrix of $M$ consisting of the rows corresponding the indices in $I$, and where the rows of $M'$ have been rearranged accordingly for notational convenience.

\subsection{Supertorus Orbits}
Using the above coordinates on $\QGr(r,n)$, the left action of $Q(1)^n$ is via matrix multiplication, where $Q(1)^n(A)$ consists of diagonal matrices

\begin{align*}
    \begin{pmatrix}
        t_1(1+\xi_1) && \\
        & \ddots & \\
        && t_n(1+\xi_n)
    \end{pmatrix}
\end{align*}
for $t_i \in A_\0^\times$ and $\xi_i \in A_\1$. As a result, the affine charts of a supertorus orbit closure will have coordinates of the forms $t_it_j^{-1}(1-\xi_i\xi_j)$ and $t_i t_j^{-1}(\xi_i-\xi_j)$. Sometimes we will act by a subsupertorus $Q(1)^d \subseteq Q(1)^n$, in which case some coordinates will have the forms $t_i$ and $t_i \xi_i$.

The Lie superalgebra of $Q(1)^n$ can be written as $\q(1)^n = \Span_\CC\{x_1, ..., x_n, \theta_1, ..., \theta_n\}$ such that $[\theta_i, \theta_j] = \delta_{ij}x_i$.

\subsection{Polytopes and the Momentum Map}

The classical type-A Grassmannian $\Gr(r,n)$ admits a momentum map to $\RR^n$ whose image is the $(r,n)$-hypersimplex $\Delta_{r,n}$, the convex hull of the vectors in $\{0,1\}^n$ with exactly $r$ instances of 1. Here, the target space $\RR^n$ is really the dual space of $t_\RR$, a real form of the Lie algebra of a maximal torus $T \subseteq \GL(n) \to \Gr(r,n)$. 

Remarkably, the momentum map image of a torus orbit closure (i.e.\ a toric variety) in $\Gr(r,n)$ is a subpolytope of $\Delta_{r,n}$ whose normal fan is the fan of the toric variety (see e.g.\ \cite{GGMS}). The edges and vertices of such a subpolytope must be a subset of those of $\Delta_{r,n}$. A similar phenomenon occurs in the super case for the action of $Q(1)^n$ on $\QGr(r,n)$ (see \cite{Jankowski2}), but this will not be needed in the present paper.

\subsection{Standard Extensions}
In this section we introduce an embedding of Grassmannians (see e.g.\ \cite{PenkTikh}) that will be a useful tool in Lemmas \ref{lemma:PicQGr=0} and \ref{lemma:PiPicQGr=0}.
\begin{definition}
    A \textit{standard extension} of Grassmannians is an embedding of the form
    \begin{align*}
        \Gr(r,n) &\to \Gr(r',n') \\
        \{V \subseteq \CC^n\} &\mapsto \{V \oplus W \subseteq \CC^{n} \oplus \CC^{n'-n}\}
    \end{align*}
    where $W$ is a fixed $(r'-r)$-dimensional subspace of $\CC^{n'-n}$.
\end{definition}

Standard extensions generalize easily to the super setting. In the type-Q case, the map is given by taking the direct sum with a fixed isomeric $(r'-r|r'-r)$-dimensional subspace $W$ of $\CC^{n'-n|n'-n}$. Functorially, the map on $A$-points can be written as
\begin{align*}
    M \mapsto \begin{pmatrix}
        M & 0_{n \times (r' -r)} \\
        0_{(r'-r) \times r} & 1_{(r'-r) \times (r'-r)} \\
        0_{(n' - n - r'+r) \times r} & 0_{(n' - n - r'+r) \times (r'-r)}
    \end{pmatrix}.
\end{align*}

\section{Picard Groups}\label{sec:Pic}

\subsection{Picard Group of the Isomeric Supergrassmannian}

Let $X = \QGr(r,n)$ for some integers $n \geq r \geq 0$.

\begin{lemma}\label{lemma:PicPPi1=Z}
    If $r=1$ and $n=2$, so that $X=\PP_\Pi^1$, then $\Pic(X) = \ZZ$.
\end{lemma}

\begin{proof}
    Since the odd dimension of $\PP_\Pi^1$ is 1, we have $\O_{\PP_\Pi^1, \0} = \O_{\PP^1}$ and hence $\Pic(\PP_\Pi^1) = \Pic(\PP^1) = \ZZ$.
\end{proof}

Recall that $X_\red \cong \Gr(r,n)$ is the underlying variety of $X \cong \QGr(r,n)$, given by quotienting the structure sheaf by the ideal $\J_X = i_X^{-1}((\K_{X,\1}))$.

\begin{lemma}\label{lemma:RestrictionMapInjectiveQGr}
    The restriction map $\Pic(X) \to \Pic(X_\red)$ is injective.
\end{lemma}

\begin{proof}
    Consider the sequence of successive quotients
    $$\O_{X,\0} \to \O_{X,\0}/(\J_X^{2k})_\0 \to \O_{X,\0}/(\J_X^{2(k-1)})_\0 \to ... \to \O_{X,\0}/(\J_X^4)_\0 \to \O_{X,\0}/(\J_X^2)_\0$$
    where $k = \lfloor \frac{n}{2} \rfloor$, and note that $\J_X^{2(k+1)} = 0$. Write $X_i$ for the scheme $(|X|, \O_{X,\0}/(\J_X^{2i})_\0)$. For $i=1, ..., k$, we have an exact sequence of sheaves of abelian groups $$0 \to (\J_X^{2i})_\0 / (\J_X^{2(i+1)})_\0 \to \left( \O_{X,\0}/(\J_X^{2(i+1)})_\0 \right)^\times \to \left( \O_{X,\0}/(\J_X^{2i})_\0 \right)^\times \to 0$$
    in which the first nonzero map is given by $f \mapsto 1+f$. This yields the long exact sequence in cohomology
    $$... \to H^1(X, (\J_X^{2i})_\0 / (\J_X^{2(i+1)})_\0) \to \Pic(X_{i+1}) \to \Pic(X_i) \to ...$$
    where we have used the isomorphism $\Pic(X) \cong H^1(X,\O_X^\times)$.
    
    We claim that $\Pic(X_{i+1}) \to \Pic(X_i)$ is injective, for which it suffices to show that $H^1(X, (\J_X^{2i})_\0 / (\J_X^{2(i+1)})_\0) = 0$. We first observe that $(\J_X^{2i})_\0 / (\J_X^{2(i+1)})_\0 \cong \J_X^{2i} / \J_X^{2i+1}$. It is seen in \cite{Noja} that $\J_X / \J_X^2 \cong \Omega_{X_\red}^1$, and indeed $\J_X^{2i}/\J_X^{2i+1} \cong \Omega^{2i}_{X_\red}$. But the Hodge numbers $h^q(X_{\red}, \Omega_{X_{\red}}^{p})$ are nonzero only if $p=q$, so $H^1(X, (\J_X^{2i})_\0 / (\J_X^{2(i+1)})_\0) = 0$. Therefore $$\Pic(X) \to \Pic(X_k) \to ... \to \Pic(X_1) \to \Pic(X_\red)$$ is injective, and we are finished.
\end{proof}

\begin{lemma}\label{lemma:PicQGr=0}
    If $0<r<n-1$, then $\Pic(X) = 0$.
\end{lemma}

\begin{proof}
    Consider the commutative diagram
\begin{center}
\begin{tikzcd}
    \PP^2 \arrow{r} \arrow{d} & \Gr(r,n) \arrow{d} \\
    \PP_\Pi^2 \arrow{r} & \QGr(r,n)
\end{tikzcd}
\end{center}
where the vertical arrows are the natural closed immersions and the horizontal arrows are standard extensions. Applying the Picard group functor, we obtain
\begin{center}
\begin{tikzcd}
    \ZZ & \ZZ \arrow{l}{\sim} \\
    \Pic(\PP_\Pi^2) \arrow[hookrightarrow]{u} & \Pic(\QGr(r,n)) \arrow[hookrightarrow]{u}, \arrow{l}
\end{tikzcd}
\end{center}
so it suffices to show that $\Pic(\PP_\Pi^2) = 0$.

Using the description of $X = \PP_\Pi^2$ as a toric supervariety with supertorus $Q(1)^2$, it is straightforward to compute the even coordinates on the affine charts:
\begin{align*}
    \CC[U_0]_\0 &= \CC[t_1, t_2, t_1t_2 \xi_1 \xi_2] \\
    \CC[U_1]_\0 &= \CC[t_1^{-1}, t_1^{-1}t_2(1+\xi_1\xi_2), t_1^{-2}t_2 \xi_1 \xi_2] \\
    \CC[U_2]_\0 &= \CC[t_2^{-1}, t_1t_2^{-1}(1-\xi_1\xi_2), t_1t_2^{-2} \xi_1 \xi_2] \\
    \CC[U_{01}]_\0 &= \CC[t_1^{\pm 1}, t_2, t_2 \xi_1 \xi_2] \\
    \CC[U_{02}]_\0 &= \CC[t_1, t_2^{\pm 1}, t_1 \xi_1 \xi_2] \\
    \CC[U_{12}]_\0 &= \CC[t_1^{-1}, (t_1^{-1}t_2(1+\xi_1\xi_2))^{\pm 1}, t_1^{-1} \xi_1 \xi_2]
\end{align*}
Now from the exact sequence
$$0 \to \O_X^\times \to \K_X^\times \to \K_X^\times / \O_X^\times \to 0$$
we obtain
$$H^0(\K_X^\times) \to H^0(\K_X^\times / \O_X^\times) \to \Pic(X) \to 0$$
since $\K_X^\times$ is a constant sheaf. Using the trivialization over the open cover $\{U_0, U_1, U_2\}$, we note that $$H^0(\K_X^\times / \O_X^\times) = \{(f_0, f_1, f_2) \mid f_i \in \CC(U_i)^\times / \CC[U_i]^\times, f_i/f_j \in \CC[U_{ij}]^\times\}.$$
Hence $\Pic X$ consists of the data
\begin{align*}
    \frac{f_1}{f_0} &= t_1^\ell \cdot \left( 1 + \sum_{m_2>0} b_m t^m \xi_1 \xi_2 \right) \\
    \frac{f_2}{f_0} &= t_2^\ell \cdot \left( 1 + \sum_{m_1>0} c_m t^m \xi_1 \xi_2 \right)
\end{align*}
for $\ell \in \ZZ$ and $b_m, c_m \in \CC$, such that
\begin{align*}
    \frac{f_2}{f_1} &= (t_1^{-1}t_2(1+\xi_1\xi_2))^\ell \cdot \left(1+... \right).
\end{align*}
An explicit computation reveals that this is not possible unless $\ell=0$, meaning the image of $\Pic(\PP_\Pi^2) \to \Pic(\PP^2) \cong \ZZ$ is 0, as desired.
\end{proof}

Accumulating the results of the Lemmas \ref{lemma:PicPPi1=Z} and \ref{lemma:PicQGr=0} together with the cases of $X$ a point, we obtain:
\begin{proposition}\label{prop:PicQGr}
    Let $0 \leq r \leq n$. Then $$\Pic(\QGr(r,n)) = \begin{cases}
        \ZZ & r=1, n=2 \\
        0 & \text{otherwise}
    \end{cases}$$
\end{proposition}

\subsection{Picard Groups of Toric Subvarieties}
In this section we use essentially the same techniques as in the prior section to compute the Picard group of toric supervarieties in $\QGr(r,n)$.

Recall that we write $\q(1)^n = \Span_\CC \{x_1, ..., x_n, \theta_1, ..., \theta_n\}$, where $[\theta_i, \theta_j] = \delta_{ij} x_i$, for the Lie superalgebra of the supertorus $Q(1)^n$.

Let $y \in |\QGr(r,n)|$ be a closed point. As in \cite{Jankowski2}, its stabilizer is
$$\Stab_{q(1)^n} y = \Span_\CC \left\{ \sum_{i \in I_j} x_i \;\Bigg|\; j=1, ..., s \right\} + \Span_\CC \left\{ \sum_{i \in I_j} \theta_i \;\Bigg|\; j=1, ..., s \right\}$$ for some partition $I_1 \sqcup ... \sqcup I_s = \{1, ..., n\}$. We may then let $T \subseteq Q(1)^n$ be the subgroup 
$$T=\prod_{j=1}^s \prod_{i \in I_j \backslash \max(I_j)} Q(1)$$
so that the orbit closure $X = \cl(Q(1)^n \cdot y)$ is a faithful toric supervariety for $T$. Its Lie superalgebra is
$$\t = \Span_\CC\{x_i, \theta_i \mid i \neq \max(I_j) \text{ for any } j=1, ..., s\},$$
and we write $\t_\RR$ for the $\RR$-linear span of the same basis vectors. Then, since stabilizers of points are sums of unit vectors, it follows that the decorated fan will be a complete fan on rays of the form $\pm \RR_+ (x_{i_1} + ... + x_{i_k})$ decorated by $\CC(\theta_{i_1} + ... + \theta_{i_k})$.

In particular, the fermionic sheaf $\F_X \cong \Omega_{X_\red}^1$ as in the case of $X=\QGr(r,n)$. This can also be seen at the level of affine toric charts of $X$, each of which is generated by a collection of even weight vectors that reduce to the ordinary even generators, as well as one odd generator in each of the same weight spaces. We then have:

\begin{lemma}\label{lemma:RestrictionMapInjectiveToric}
    The restriction map $\Pic(X) \to \Pic(X_\red)$ is injective.
\end{lemma}
\begin{proof}
    A result of Danilov (\cite[Corollary 12.7]{Danilov}) yields $H^1(X, \Omega^q_X) = 0$ for $q>1$, so the argument of Lemma \ref{lemma:RestrictionMapInjectiveQGr} applies.
\end{proof}

Let $P$ be the polytope corresponding to $X_\red$, and note that such a polytope induces a unique toric supervariety $X_P \subseteq \QGr(r,n)$, up to isomorphism. In particular, $X_P$ may be described by a decorated polytope which is uniquely determined by $P$, and equals the image of $X_P$ under an ``odd momentum map" $\QGr(r,n) \to \Pi \uq(n)^*$ (see \cite{Jankowski2}). Considering $P$ as a subset of $(\t_\RR)_\0^*$, its edges are parallel to either $x_i^*$ or $x_i^*-x_j^*$ for some $i,j \in \{1, ..., n\}$. The fan $\Sigma_P$ dual to $P$ consists of rays of the form $\pm \RR_+ (x_{i_1} + ... + x_{i_k})$ as above.

\begin{definition}\label{def:pcp}
    The \textit{parallel component number} $pc(P)$ of a polytope $P \subseteq \RR^n$ is any of the following equivalent notions:
    \begin{enumerate}        
        \item The number of lines $L \in \PP(\RR^n)$ such that the 1-skeleton of $P$ becomes disconnected upon removal of the edges parallel to $L$.

        \item The number of 2-colorings of the vertices of $P$ such that all the edges between differently colored vertices are parallel.

        \item The number of codimension-1 hyperplanes $H \subseteq (\RR^n)^*$ such that $\{\sigma \cap H \mid \sigma \in \Sigma+P\}$ is a subfan of $\Sigma_P$.

        \item The number of surjective toric morphisms $X_P \to \PP_\Pi^1$, up to automorphism of $\PP_\Pi^1$.
    \end{enumerate}
\end{definition}

\begin{example}
    The supervarieties $\PP_\Pi^1 \times \PP_\Pi^2$ and $\PP_\Pi^1 \times \PP_\Pi^1 \times \PP_\Pi^1$ are described by a triangular prism and a cube, respectively. These polytopes have parallel component numbers 1 and 3, respectively, as seen in Figure \ref{fig:triangleAndCube}.  In view of the forthcoming Proposition \ref{prop:PicTSV}, their respective Picard groups are $\ZZ$ and $\ZZ^3$.
\end{example}

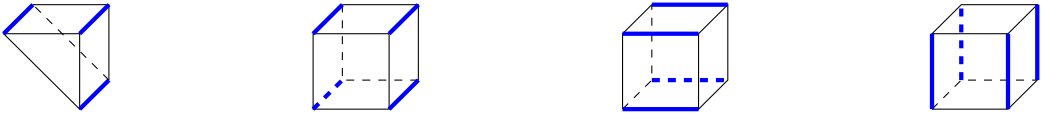
\begin{figure}
    \centering
    \begin{subfigure}{.25\textwidth}
        \centering
        \begin{tikzpicture}
        % Define the coordinates of the triangle
        \coordinate (A) at (1,1,1);
        \coordinate (B) at (1,0,1);
        \coordinate (C) at (0,1,1);
        \coordinate (A') at (1,1,0);
        \coordinate (B') at (1,0,0);
        \coordinate (C') at (0,1,0);
        
        % Draw the black edges of the triangle
        \draw[black] (A) -- (B) -- (C) -- cycle;
        \draw[black] (B') -- (A') -- (C');
        \draw[black, dashed] (B') -- (C');
        
        % Draw the interval (here, it's a segment on the z-axis)
        \draw[blue, ultra thick] (A) -- (A');
        \draw[blue, ultra thick] (B) -- (B');
        \draw[blue, ultra thick] (C) -- (C');
        \end{tikzpicture}
    \end{subfigure}
    ~
    ~
    ~
    \begin{subfigure}{.25\textwidth}
        \centering
        \begin{tikzpicture}
        % Define the coordinates of the triangle
        \coordinate (A) at (1,1,1);
        \coordinate (B) at (1,0,1);
        \coordinate (C) at (0,1,1);
        \coordinate (D) at (0,0,1);
        \coordinate (A') at (1,1,0);
        \coordinate (B') at (1,0,0);
        \coordinate (C') at (0,1,0);
        \coordinate (D') at (0,0,0);
        
        % Draw the black edges of the triangle
        \draw[black] (A) -- (B) -- (D) -- (C) -- cycle;
        \draw[black] (B') -- (A') -- (C');
        \draw[black, dashed] (B') -- (D') -- (C');
        
        % Draw the interval (here, it's a segment on the z-axis)
        \draw[blue, ultra thick] (A) -- (A');
        \draw[blue, ultra thick] (B) -- (B');
        \draw[blue, ultra thick] (C) -- (C');
        \draw[blue, ultra thick, dashed] (D) -- (D');
        \end{tikzpicture}
    \end{subfigure}
    ~
    ~
    ~
    \begin{subfigure}{.25\textwidth}
        \centering
        \begin{tikzpicture}
        % Define the coordinates of the triangle
        \coordinate (A) at (1,1,1);
        \coordinate (B) at (1,0,1);
        \coordinate (C) at (0,1,1);
        \coordinate (D) at (0,0,1);
        \coordinate (A') at (1,1,0);
        \coordinate (B') at (1,0,0);
        \coordinate (C') at (0,1,0);
        \coordinate (D') at (0,0,0);
        
        % Draw the black edges of the triangle
        \draw[black] (A) -- (B) -- (B') -- (A') -- cycle;
        \draw[black] (D) -- (C) -- (C');
        \draw[black, dashed] (D) -- (D') -- (C');
        
        % Draw the interval (here, it's a segment on the z-axis)
        \draw[blue, ultra thick] (C) -- (A);
        \draw[blue, ultra thick] (D) -- (B);
        \draw[blue, ultra thick] (C') -- (A');
        \draw[blue, ultra thick, dashed] (D') -- (B');
        \end{tikzpicture}
    \end{subfigure}
    ~
    ~
    ~
    \begin{subfigure}{.25\textwidth}
        \centering
        \begin{tikzpicture}
        % Define the coordinates of the triangle
        \coordinate (A) at (1,1,1);
        \coordinate (B) at (1,0,1);
        \coordinate (C) at (0,1,1);
        \coordinate (D) at (0,0,1);
        \coordinate (A') at (1,1,0);
        \coordinate (B') at (1,0,0);
        \coordinate (C') at (0,1,0);
        \coordinate (D') at (0,0,0);
        
        % Draw the black edges of the triangle
        \draw[black] (A) -- (C) -- (C') -- (A') -- cycle;
        \draw[black] (D) -- (B) -- (B');
        \draw[black, dashed] (D) -- (D') -- (B');
        
        % Draw the interval (here, it's a segment on the z-axis)
        \draw[blue, ultra thick] (B) -- (A);
        \draw[blue, ultra thick] (D) -- (C);
        \draw[blue, ultra thick] (B') -- (A');
        \draw[blue, ultra thick, dashed] (D') -- (C');
        \end{tikzpicture}
    \end{subfigure}
    
    \caption{One parallel component in a triangular prism; three in a cube}
    \label{fig:triangleAndCube}
\end{figure}

\begin{proposition}\label{prop:PicTSV}
    $\Pic(X_P) \cong \ZZ^{pc(P)}$, a saturated subgroup of $\Pic((X_P)_\red)$
\end{proposition}

\begin{proof}
    An element of $\Pic(X) \cong H^1(\O_{X,0}^\times)$ contains the data of a transition function from each maximal affine toric open to each other with which it shares a maximal proper subset that is an affine toric open, i.e.\ a codimension-1 cone in the language of fans. This can be encoded by a transition function from each vertex of $P$ to each adjacent vertex, along the incident edge. Since the 1-skeleton of $P$ is connected, the data of these transition functions conversely determine a unique element of $\Pic(X)$.
    
    A transition function $f_{vw}$ from $U_v$ to $U_w$ is a unit in $\O_{X,\0}(U_v \cap U_w)$, and they are subject to cocycle conditions $f_{v_1v_2} f_{v_2v_3} \cdots f_{v_{k-1}v_k} f_{v_k v_1} = 1$. By Lemma \ref{lemma:RestrictionMapInjectiveToric}, we may write the transition functions in the form $t_i^\ell$ or $(t_it_j^{-1}(1-\xi_i\xi_j))^\ell$ for some $\ell \in \ZZ$, depending on whether the relevant edge is parallel to $x_i^*$ or $x_i^* - x_j^*$.
    
    Applying the cocycle condition, we see that the only way to cancel a $(1-\xi_i \xi_j)^\ell$ term is with an opposite $(1-\xi_i \xi_j)^{-\ell}$ term. Hence $\ell=0$ unless the cycle contains another edge parallel to $x_i^*-x_j^*$. It then follows that $t_i^\ell$ can be canceled out only if the cycle contains another edge parallel to $x_i^*$. In particular, a transition function along an edge $e$ can be nonzero only if every cycle containing $e$ contains another edge parallel to $e$. The number of such directions is counted by part (a) of Definition \ref{def:pcp}

    Conversely, for each direction counted by $pc(P)$, the choice of $\ell \in \ZZ$ is free and independent of other such choices. These choices yield the Picard group $\Pic(X_P)$, as desired.
\end{proof}

The same conclusion applies to the $Q(1)^n$ toric supervariety of any complete fan whose rays are of the form $\pm \RR_+ (x_{i_1} + ... + x_{i_k})$, decorated by $\theta_{i_1}+...+\theta_{i_k}$. However, we caution that not every fan of this form can admit these decorations. For example, the rays $\RR_+(x_1+x_2)$ and $\RR_+(x_1+x_3)$ cannot together form a 2-cone because $[\theta_1+\theta_2, \theta_1+\theta_3] \notin \Span_\CC\{x_1+x_2, x_1+x_3\}$. In terms of the polytope, the edges must be parallel to $x_i^*$ or $x_i^* - x_j^*$.

\begin{example}\label{ex:PicPentagon}
    Let $X$ be the blowup of $\PP^1 \times \PP^1$ at a torus-fixed point, so the corresponding polytope $P$ is a pentagon as in Figure \ref{fig:pentagon}. If the rays of $\Sigma_P$ are appropriately decorated by $\theta_1, \theta_2,$ and $\theta_1+\theta_2$, then the Picard group of the resulting toric supervariety $X_P$ is isomorphic to $\ZZ^2$.

    Note that $X_P$ is not a blowup of $\PP_\Pi^1 \times \PP_\Pi^1$; any new rays introduced by a blowup operation on a toric supervariety must be decorated by 0.
\end{example}

\begin{remark}\label{rem:PicTSV}
    In the case that $X_P$ does arise as a supertorus orbit closure in $\QGr(r,n)$, then the argument of the forthcoming Lemma \ref{lemma:PiPic=0IfNoSimplexFactor} shows that each contribution to $pc(P)$ arises from a factor of $\PP_\Pi^1$ in $X_P$. Due to the isomorphism of $\Pic(\PP_\Pi^1) \cong H^1(O_{X,\0}^\times) \cong \Pic(\PP^1)$, it follows that any invertible sheaf on $X_P$ can be understood as a tensor product of a choice of invertible sheaf $\O(\ell_i)$ on each factor of $\PP_\Pi^1$. That is, if $X_P \cong (\PP_\Pi^1)^d \times Y$ for $d$ maximal, then an invertible sheaf on $X_P$ is determined entirely by its restriction to any closed subvariety $(\PP^1)^d \times pt$.
\end{remark}

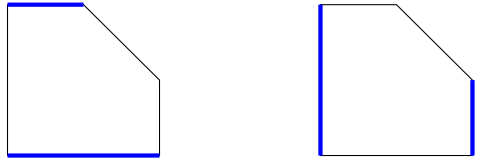
\begin{figure}
    \begin{subfigure}{.25\textwidth}
        \centering
        \begin{tikzpicture}
        \coordinate (A) at (0,0);
        \coordinate (B) at (2,0);
        \coordinate (C) at (2,1);
        \coordinate (D) at (1,2);
        \coordinate (E) at (0,2);
        
        \draw[black] (A) -- (B) -- (C) -- (D) -- (E) -- cycle;
        
        \draw[blue, ultra thick] (A) -- (B);
        \draw[blue, ultra thick] (D) -- (E);
        \end{tikzpicture}
    \end{subfigure}
    ~
    ~
    ~
    \begin{subfigure}{.25\textwidth}
        \centering
        \begin{tikzpicture}
        \coordinate (A) at (0,0);
        \coordinate (B) at (2,0);
        \coordinate (C) at (2,1);
        \coordinate (D) at (1,2);
        \coordinate (E) at (0,2);
        
        \draw[black] (A) -- (B) -- (C) -- (D) -- (E) -- cycle;
        
        \draw[blue, ultra thick] (A) -- (E);
        \draw[blue, ultra thick] (B) -- (C);
        \end{tikzpicture}
    \end{subfigure}
    \caption{Two parallel components in the pentagon of Example \ref{ex:PicPentagon}}
    \label{fig:pentagon}
\end{figure}

\section{{$\Pi$}-Picard Sets}\label{sec:PiPic}

If $X$ is a supervariety, write $X/\J^2$ for its closed subvariety $(|X|, \O_X/\J_X^2)$. By removing the grading, since $\O_X/\J_X^2$ is a sheaf of \textit{commutative} rings, we may view $X/\J^2$ as an honest (non-super) scheme, which we call $X_\thick$. Observe that $\Pic_\Pi(X/\J^2) = H^1(X, (\O_X/\J_X^2)^\times) = \Pic(X_\thick)$ and $\Pic_\Pi(X_\red) = H^1(X, (\O_X/\J_X)^\times) = \Pic(X_\red)$.

\subsection{{$\Pi$}-Picard Set of the Isomeric Supergrassmannian}

Once again, let $X=\QGr(r,n)$ for some integers $n \geq r \geq 0$.

\begin{lemma}\label{lemma:RestrictionMapFactorsQGr}
    The restriction map $\Pic_\Pi(X) \to \Pic_\Pi(X_\red)$ factors as an injection $\Pic_\Pi(X) \to \Pic_\Pi(X/\J^2)$ followed by a surjection $\Pic_\Pi(X/\J^2) \to \Pic_\Pi(X_\red)$.
\end{lemma}

\begin{proof}
    Injectivity of $\Pic_\Pi(X) \to \Pic_\Pi(X/\J^2)$ is the same as in Lemma \ref{lemma:RestrictionMapInjectiveQGr}. Likewise, surjectivity of $\Pic_\Pi(X/\J^2) \to \Pic_\Pi(X_\red)$ follows from $H^2(\Gr(r,n), \Omega^1) = 0$.
\end{proof}

This lemma has two important consequences. For one, we now have a frame of reference in which to describe the cohomology set $\Pic_\Pi(X)$; that is, we can describe it as a subset of the group $\Pic_\Pi(X/\J^2) = \Pic(X_\thick)$. Secondly, it yields the following lemma.

\begin{lemma}\label{lemma:PiPicQGrThick=Z+C}
    If $0<r<n$, then $\Pic_\Pi(X/\J^2) \cong \ZZ \oplus \CC$. In particular, $\Pic_\Pi(\PP_\Pi^1)=\ZZ \oplus \CC$.
\end{lemma}
\begin{proof}
    From the exact sequence $$0 \to \Omega^1 \to (\O_X/\J_X^2)^\times \to (\O_X/\J_X)^\times \to 0$$
    and from surjectivity of $H^0((\O_X/\J_X^2)^\times) \to H^0((\O_X/\J_X)^\times)$, we obtain the exact sequence
    \begin{center}
    \begin{tikzcd}[row sep=tiny]
        0 \arrow[r] & H^1(\Omega^1) \arrow[r] & \Pic(X_\thick) \arrow[r] & \Pic(X_\red) \arrow[r] & 0 \\
        & \CC \arrow[u,equal] & & \ZZ \arrow[u,equal] &
    \end{tikzcd}
    \end{center}
    and hence $\Pic_\Pi(X/\J^2) = \Pic(X_\thick) \cong \ZZ \oplus \CC$.
\end{proof}

It is then straightforward to see that for $X = \PP_\Pi^n$, the transition functions for a $\Pi$-invertible sheaf on $X/\J^2$ take the form
\begin{align*}
    f_i f_0^{-1} = t_i^\ell(1+c\xi_i)
\end{align*}
for some $\ell \in \ZZ$ and $c \in \CC$. To see which transition functions also work for $X$ itself, we calculate
\begin{align*}
    f_i f_j^{-1} &= (f_i f_0^{-1})(f_j f_0^{-1})^{-1} \\
    &= (t_i t_j^{-1})^\ell (1 + c(\xi_i - \xi_j) - c^2 \xi_i \xi_j)
\end{align*}
so that the only requirement is $c^2 = \ell$. Therefore

\begin{lemma}\label{lemma:PiPicPPi^n=ZcupZ}
    If $n>1$, then $\Pic_\Pi(\PP_\Pi^n) = \{(\ell,c) \in \ZZ \oplus \CC \mid c^2 = \ell\}$.
\end{lemma}

The only difference between the elements $(\ell, +\sqrt{\ell})$ and $(\ell, -\sqrt{\ell})$ is that the odd involution is negated; the sheaves themselves are isomorphic.

\begin{remark}\label{rem:PiPicQGr}
    We write $\O_\Pi(\ell)$ for the sheaf represented by $(\ell, \pm \sqrt\ell)$. Naturally, $\O_\Pi(1)$ is the tautological sheaf whose fiber at the point $W \in \PP_\Pi^n$ is $W^*$, and $\O_\Pi(-1)$ is the dual sheaf with fiber $W$. We observe that all $\O_\Pi(\ell)$ are symmetric powers of these. For $\ell>0$,
    \begin{align*}
        \O_\Pi(\ell) &\cong S^\ell \O_\Pi(1) \\
        \O_\Pi(-\ell) &\cong S^{\ell} \O_\Pi(-1)
    \end{align*}
    and so their fibers at $W$ are $S^\ell W^*$ and $S^{\ell} W$, respectively. Note however that the odd involution must be rescaled appropriately following this operation.

    If an isomeric basis of $\CC^{n+1|n+1}$ is chosen, as must be done when we fix a supertorus $Q(1)^{n+1} \subseteq Q(n+1)$, then the $\O_\Pi(\ell)$ can also be defined in terms of the standard representation $V = \CC^{n+1|n+1}$ of $Q(n+1)$. Fix a closed point $W \in \PP_\Pi^n$, and let $W^\perp$ be its orthogonal complement with respect to the given basis. Then $Q(W) \times Q(W^\perp) \cong Q(1) \times Q(n)$ is a block-diagonal subgroup of $Q(V) \cong Q(n+1)$, and we can ask for the highest weight spaces of $S^\ell W^*$ and $S^\ell W$ for the action of $Q(W)$. For $\ell>0$, these highest weight spaces assemble into the sheaves $\O_\Pi(\ell)$ and $\O_\Pi(-\ell)$, respectively.
\end{remark}

\begin{lemma}\label{lemma:PiPicQGr=0}
    If $1 < r < n-1$, then $\Pic_\Pi(X) = 0$.
\end{lemma}
\begin{proof}
    We employ the same strategy that was used in Lemma \ref{lemma:PicQGr=0}. It is straightforward to see that a standard extension $\QGr(2,4) \to \QGr(r,n)$ induces an isomorphism on $\Pi$-Picard sets at the level of $X/\J^2$. That is, we have a commutative diagram
    \begin{center}
    \begin{tikzcd}
        \ZZ \oplus \CC & \ZZ \oplus \CC \arrow{l}{\sim} \\
        \Pic_\Pi(\QGr(2,4)) \arrow[hookrightarrow]{u} & \Pic_\Pi(\QGr(r,n)) \arrow[hookrightarrow]{u}, \arrow{l}
    \end{tikzcd}
    \end{center}
    so it suffices to show that $\Pic_\Pi(\QGr(2,4)) = 0$.
    
    Choose two embeddings $\PP_\Pi^2 \to \QGr(2,4)$ whose momentum map images are adjacent triangles in the hypersimplex $\Delta_{2,4}$. As above, these embeddings induce isomorphisms $\Pic_\Pi((\QGr(2,4))_\thick) \to \Pic_\Pi((\PP_\Pi^2)_\thick)$. Since they occur as adjacent triangles, the clockwise directions within the two triangles disagree on their shared edge. (Later, in the proof of Lemma \ref{lemma:PiPic=0IfTwoTrianglesMeet}, we will illustrate this phenomenon more explicitly.) Thus we may assume without loss of generality that these two maps $\ZZ \oplus \CC \to \ZZ \oplus \CC$ are given by $\pm 1$. It follows that an element $(\ell,c) \in \Pic_\Pi(\QGr(2,4)) \subseteq \ZZ \oplus \CC$ must satisfy $c^2 = \ell$ and $(-c)^2 = -\ell$, so indeed $\Pic_\Pi(\QGr(2,4)) = 0$.
\end{proof}

We now accumulate the results of Lemmas \ref{lemma:PiPicQGrThick=Z+C}, \ref{lemma:PiPicPPi^n=ZcupZ}, and \ref{lemma:PiPicQGr=0} in the following proposition. Note in particular that the only nonzero cases are $\PP_\Pi^n$.

\begin{proposition}\label{prop:PiPicQGr}
    Let $0 \leq r \leq n$. Then
    $$\Pic_\Pi(\QGr(r,n)) = \begin{cases}
        \ZZ \oplus \CC & r=1, n=2 \\
        \{(\ell, c) \in \ZZ \oplus \CC \mid c^2=\ell\} &r\in \{1, n-1\}, n>2 \\
        0 & \text{otherwise}
    \end{cases}$$
\end{proposition}

\subsection{$\Pi$-Picard Sets of Toric Subvarieties}

Now let $X$ be the closure of a $Q(1)^n$-orbit in $\QGr(r,n)$, and let $P$ be the image of its underlying variety under the momentum map, a subpolytope of $\Delta_{r,n}$ whose vertices and edges are also vertices and edges of $\Delta_{r,n}$.

\begin{lemma}\label{lemma:RestrictionMapFactorsToric}
    The restriction map $\Pic_\Pi(X) \to \Pic_\Pi(X_\red)$ factors as an injection $\Pic_\Pi(X) \to \Pic_\Pi(X/\J^2)$ followed by a surjection $\Pic_\Pi(X/\J^2) \to \Pic_\Pi(X_\red)$.
\end{lemma}

\begin{proof}
    As in Lemma \ref{lemma:RestrictionMapFactorsQGr}, injectivity of $\Pic_\Pi(X) \to \Pic_\Pi(X/\J^2)$ is the same as in Lemma \ref{lemma:RestrictionMapInjectiveToric}. However, unlike in Lemma \ref{lemma:RestrictionMapFactorsQGr}, it is not necessarily the case that $H^2(X, \Omega^1)=0$. Instead, we employ a slightly more manual argument to show that $\Pic_\Pi(X/\J^2) \to \Pic_\Pi(X_\red)$ is surjective.

    Since the even part of $\O_X/\J_X^2$ is exactly $\O_{X_\red}$, there is a splitting $X/\J^2 \to X_\red$ of the closed immersion $X_\red \to X/\J^2$. Since the composition of these maps is the identity on $X_\red$, the same is true for the corresponding maps on $\Pic_\Pi(X_\red)$.
\end{proof}

As before, this implies that $\Pic_\Pi(X/\J^2) \cong \Pic(X_\red) \oplus H^1(\Omega^1)$. Of course, each transition function along an edge of $P$ consists of an integer and a complex number. These two pieces of data are intimately linked; indeed, one should expect that the continuous part of $\Pic_\Pi(X/\J^2)$ has the same rank as $\Pic(X_\red)$. It is therefore no coincidence that $\rk \Pic(X_\red) = \dim H^1(\Omega^1)$. See e.g.\ \cite{CLS} for the case of $P$ a simple polytope. 

\begin{lemma}
    It holds that $\Pic_\Pi(X/\J^2) \cong (\ZZ \oplus \CC)^{\rk \Pic(X_\red)}$.
\end{lemma}

We record the following fact about 2-faces of $P$.

\begin{lemma}\label{lemma:FaceIsTriangleOrSquare}
    A 2-face of $P$ is either a triangle or a square.
\end{lemma}
\begin{proof}
    It is a straightforward exercise to verify that any set of more than three coplanar vertices in $\Delta_{r,n}$ are the vertices of a square.
\end{proof}

\begin{lemma}\label{lemma:PiPic=0IfTwoTrianglesMeet}
    If two triangular faces in $P$ meet at an edge and are not part of a larger simplicial face, then any $\Pi$-invertible sheaf trivializes on the union of the corresponding cells of $X$.
\end{lemma}

\begin{proof}
    The same proof as in Lemma \ref{lemma:PiPicQGr=0} works here as well. Since the transition functions are nicer on toric supervarieties than on $\QGr(r,n)$, we elect to include some more explicit computations in Figure \ref{fig:TwoTriangles}.
    
    Without loss of generality, two such triangles occur at the vertices $1100$, $1010$, $0110$, and $0011$, and transition functions between vertices are as pictured in Figure \ref{fig:TwoTriangles}, for some $n \in \ZZ$ and $a \in \CC$. In the simplest situation, these transition functions arise by considering the closure of a $Q(1)^4$ orbit of a generic point in $\QGr(2,4)$.
    
    The vertical downward arrow must be $(t_1t_2^{-1})^\ell(1+c(\xi_1-\xi_2) - \ell\xi_1\xi_2)$. If this equals the product of the two transition functions on the left, then $c^2=\ell$. However, if it equals the product of those on the right, then $c^2=-\ell$. Therefore $\ell=c=0$, so the $\Pi$-invertible sheaf trivializes over these cells.
\end{proof}

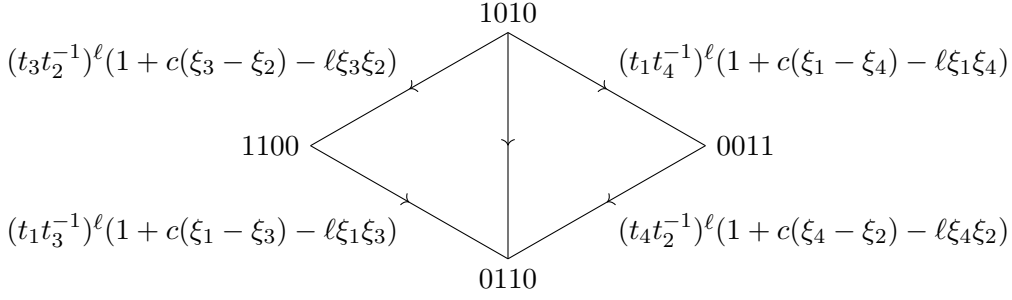
\begin{figure}
    \centering
    \begin{tikzpicture}[decoration={
    markings,
    mark=at position 0.5 with {\arrow{>}}}
    ] 
    
    % Define triangle side length
    \def\L{3}
    
    % Coordinates of the shared edge
    \coordinate (A) at (0,0);
    \coordinate (B) at (0,\L);
    
    % Third vertex of the first triangle (above)
    \coordinate (C) at ($ (A)!0.5!(B) + ({sqrt(3)/2*\L},0) $);
    
    % Third vertex of the second triangle (below)
    \coordinate (D) at ($ (A)!0.5!(B) - ({sqrt(3)/2*\L},0) $);
    
    % Draw and label triangle ABC
    \draw[postaction={decorate}] (B) -- (A);
    \draw[postaction={decorate}] (B) -- (C) node[midway, above right] {\small $(t_1t_4^{-1})^\ell(1+c(\xi_1-\xi_4) - \ell\xi_1\xi_4)$};
    \draw[postaction={decorate}] (C) -- (A) node[midway, below right] {\small $(t_4t_2^{-1})^\ell(1+c(\xi_4-\xi_2) - \ell\xi_4\xi_2)$};
    
    % Draw and label triangle ABD
    \draw[postaction={decorate}] (B) -- (D) node[midway, above left] {\small $(t_3t_2^{-1})^\ell (1+c(\xi_3-\xi_2) -\ell \xi_3 \xi_2)$};
    \draw[postaction={decorate}] (D) -- (A) node[midway, below left] {\small $(t_1t_3^{-1})^\ell(1+c(\xi_1-\xi_3) - \ell\xi_1\xi_3)$};
    
    % Shared edge A--B was already drawn
    
    % Label vertices
    \node[below] at (A) {\small $0110$};
    \node[above] at (B) {\small $1010$};
    \node[right] at (C) {\small $0011$};
    \node[left] at (D) {\small $1100$};
    
    \end{tikzpicture}

    \caption{Incompatible transition functions (unless $\ell=c=0$)}
    \label{fig:TwoTriangles}
\end{figure}

\begin{lemma}\label{lemma:PiPic=0IfNoSimplexFactor}
    If $P$ is not the product of a smaller polytope with a simplex, then $\Pic_\Pi(X)=0$.
\end{lemma}
\begin{proof}
    We prove the contrapositive. Let $S \subseteq P$ be a maximal simplicial face, and assume that there is a $\Pi$-invertible sheaf that does not trivialize over $S$. By Lemmas \ref{lemma:FaceIsTriangleOrSquare} and \ref{lemma:PiPic=0IfTwoTrianglesMeet}, all other edges incident to the vertices of $S$ must be parts of squares. In particular, $P$ consists of parallel translates of $S$, connected via parallel edges. That is, $P$ is the product of $S$ with some other polytope.
\end{proof}

From this lemma, it is clear that any nonzero contribution to $\Pic_\Pi(X)$ will occur as a result of simplex factors in $P$, or equivalently $\PP_\Pi^n$ factors in $X$. Let $I$ be the collection of these factors, and for $i \in I$ write $d_i$ for the dimension of the factor $\PP^{d_i}_\Pi$. Then $$\Pic_\Pi(X) \subseteq (\ZZ \oplus \CC)^{|I|} \subseteq (\ZZ \oplus \CC)^{\rk \Pic(X_\red)},$$ a fact that we will use in the following proposition.

\begin{proposition}\label{prop:PiPicTSV}
    It holds that $$\Pic_\Pi(X) = \{(\ell_i, c_i) \in (\ZZ \oplus \CC)^{|I|} \mid \text{at most one } c_i \neq 0, \text{and if } d_i>1, \text{then } c_i^2=\ell_i\}.$$
\end{proposition}

\begin{proof}
    The $c_i^2=\ell_i$ conditions come from Lemma \ref{lemma:PiPicPPi^n=ZcupZ}. Now for any two distinct simplex factors, say $1, 2 \in I$, choose a transition function $f_1, f_2$ along an edge of each. Then $f_1f_2 f_1^{-1} f_2^{-1} =1$, so $f_1$ and $f_2$ commute. It follows that at least one of them must be purely even, so either $c_1=0$ or $c_2=0$.
\end{proof}

\begin{remark}\label{rem:PiPicTSV}
    Observe that only one factor of $\PP_\Pi^n$ for $n>1$ can contribute nonzero transition functions to any particular $\Pi$-invertible sheaf. That is, the $\Pi$-Picard set of a product of such $\PP_\Pi^n$ is the coproduct of their $\Pi$-Picard sets, in the category of pointed sets.

    Consequently, a $\Pi$-invertible sheaf on $X \cong \prod_{i=1}^\infty (\PP_\Pi^i)^{d_i} \times Y$ for $d_i$ maximal is given by the tensor product of a $\Pi$-invertible sheaf on a single factor of $\PP_\Pi^i$ as determined in Remark \ref{rem:PiPicQGr} with an invertible sheaf on $(\PP_\Pi^1)^{d_1}$ as determined in Remark \ref{rem:PicTSV}.
\end{remark}

Previously, we indicated that the result of Proposition \ref{prop:PicTSV} can be generalized to any toric supervariety whose polytope's edges are parallel to $x_i^*$ or $x_i^* - x_j^*$, such as in Example \ref{ex:PicPentagon} and the corresponding Figure \ref{fig:pentagon}. However, analogously generalizing the result on $\Pi$-invertible sheaves is not quite so straightforward.

\begin{example}
    Let $X$ be the toric supervariety of Example \ref{ex:PicPentagon}. The transition functions of a $\Pi$-invertible sheaf on $X$ are specified in Figure \ref{fig:pentagonPicPi} such that
    \begin{align*}
    f &= t_2^{\ell_4}(1+c_4\xi_2) t_1^{\ell_3}(1+c_3\xi_1) t_2^{\ell_2}(1+c_2\xi_2) t_1^{\ell_1}(1+c_1\xi_1) \\
    &= t_1^{\ell_1+\ell_3}t_2^{\ell_2+\ell_4} (1 + (c_1+c_3)\xi_1 + (c_2+c_4)\xi_2 - (c_1c_2 + c_1c_4 - c_2c_3 + c_3c_4)\xi_1\xi_2).
    \end{align*}
    This result must satisfy the following compatibility criteria:
    \begin{align*}
        \ell_1 + \ell_2 + \ell_3 + \ell_4 &= 0 \\
        c_1+c_2+c_3+c_4 &= 0 \\
        c_1c_2 + c_1c_4 - c_2c_3 + c_3c_4 &= \ell_2+\ell_4
    \end{align*}
    For instance, the collection of $\Pi$-invertible sheaves satisfying $\ell_1+\ell_3=\ell_2+\ell_4=0$ forms a surface $(c_1+c_3)^2 + 2c_2c_3 = 0$.
\end{example}

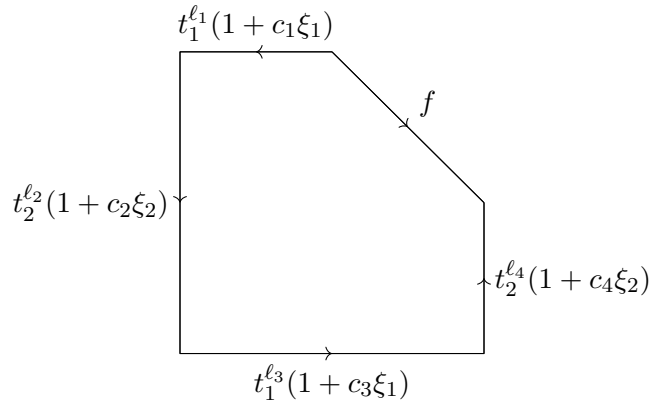
\begin{figure}
    \centering
    \begin{tikzpicture}[decoration={
    markings,
    mark=at position 0.5 with {\arrow{>}}}
    ] 
        \coordinate (A) at (0,0);
        \coordinate (B) at (4,0);
        \coordinate (C) at (4,2);
        \coordinate (D) at (2,4);
        \coordinate (E) at (0,4);
        
        \draw[black] (A) -- (B) -- (C) -- (D) -- (E) -- cycle;

        \draw[postaction={decorate}] (D) -- (E) node[midway, above] {\small $t_1^{\ell_1}(1+c_1\xi_1)$};

        \draw[postaction={decorate}] (E) -- (A) node[midway, left] {\small $t_2^{\ell_2}(1+c_2\xi_2)$};
        
        \draw[postaction={decorate}] (A) -- (B) node[midway, below] {\small $t_1^{\ell_3}(1+c_3\xi_1)$};

        \draw[postaction={decorate}] (B) -- (C) node[midway, right] {\small $t_2^{\ell_4}(1+c_4\xi_2)$};

        \draw[postaction={decorate}] (D) -- (C) node[midway, above right] {\small $f$};

    \end{tikzpicture}
    \caption{Transition functions of a $\Pi$-invertible sheaf along a pentagon}
    \label{fig:pentagonPicPi}
\end{figure}
\pagebreak

\bibliographystyle{hsiam}
\bibliography{refs}

\begin{thebibliography}{10}

\bibitem{BRP}
{\sc U.~Bruzzo, D.~Hernández~Ruipérez, and A.~Polishchuk}, {\em Notes on fundamental algebraic supergeometry. hilbert and picard superschemes}, Adv. Math., 415 (2023).
\newblock \url{https://www.sciencedirect.com/science/article/abs/pii/S0001870823000336}.

\bibitem{CN}
{\sc S.~L. Cacciatori and S.~Noja}, {\em Projective superspaces in practice}, J. Geom. Phys., 130 (2018), pp.~40--62.
\newblock \url{https://www.sciencedirect.com/science/article/pii/S0393044018301955}.

\bibitem{CLS}
{\sc D.~Cox, J.~Little, and H.~Schenck}, {\em Toric Varieties}, vol.~124 of Graduate Studies in Mathematics, American Mathematical Society, 2011.

\bibitem{Danilov}
{\sc V.~Danilov}, {\em The geometry of toric varieties}, Russ. Math. Surv., 33 (1978), pp.~97--154.
\newblock \url{https://iopscience.iop.org/article/10.1070/RM1978v033n02ABEH002305}.

\bibitem{Ewald}
{\sc G.~Ewald}, {\em Algebraic Geometry}, vol.~168 of Graduate Texts in Mathematics, Springer-Verlag, 1996.

\bibitem{GGMS}
{\sc I.~Gel'fand, R.~M. Goresky, R.~D. MacPherson, and V.~V. Serganova}, {\em Combinatorial geometries, convex polyhedra, and schubert cells}, Adv. Math., 63 (1987), pp.~301--316.
\newblock \url{https://www.sciencedirect.com/science/article/pii/0001870887900594}.

\bibitem{Jankowski}
{\sc E.~Jankowski}, {\em Toric supervarieties with one odd dimension}, Transform. Groups,  (2024).
\newblock \url{https://link.springer.com/article/10.1007/s00031-024-09889-6}.

\bibitem{Jankowski2}
{\sc E.~Jankowski}, {\em Reduced superschemes and the combinatorics of toric supervarieties}.
\newblock arXiv preprint, 2025, math/0507198.
\newblock \url{https://arxiv.org/abs/2502.15977}.

\bibitem{Kwok}
{\sc S.~Kwok}, {\em The geometry of {$\Pi$}-invertible sheaves}, J.\ Geom.\ Phys., 86 (2014), pp.~134--148.
\newblock \url{https://www.sciencedirect.com/science/article/pii/S0393044014001557}.

\bibitem{LPW}
{\sc C.~LeBrun, Y.-S. Poon, and J.~R.\ O.\~Wells}, {\em Projective embeddings of complex supermanifolds}, Commun. Math. Phys., 126 (1990), pp.~433--452.
\newblock \url{https://projecteuclid.org/journals/communications-in-mathematical-physics/volume-126/issue-3/Projective-embeddings-of-complex-supermanifolds/cmp/1104179949.full}.

\bibitem{ManinGFTCG}
{\sc Y.~I. Manin}, {\em Gauge Field Theory and Complex Geometry}, vol.~289 of Grundlehren der Mathematischen Wissenschaften, Springer-Verlag, 1988.

\bibitem{ManinTNCG}
\leavevmode\vrule height 2pt depth -1.6pt width 23pt, {\em Topics in Non-Commutative Geometry}, Porter Lectures, Princeton University Press, 1991.

\bibitem{MVP}
{\sc Y.~I. Manin, I.~B. Penkov, and A.~A. Voronov}, {\em Elements of supergeometry}, J.\ Sov.\ Math., 51 (1990), pp.~2069--2083.
\newblock \url{https://link.springer.com/article/10.1007/BF01098184}.

\bibitem{Noja}
{\sc S.~Noja}, {\em Supergeometry of {$\Pi$}-projective spaces}, J. Geom. Phys., 124 (2018), pp.~286--299.
\newblock \url{https://www.sciencedirect.com/science/article/pii/S0393044017302905}.

\bibitem{Penkov}
{\sc I.~B. Penkov}, {\em Borel-weil-bott theory of classical lie superalgebras}, J. Sov. Math., 51 (1990), pp.~2108--2140.
\newblock \url{https://link.springer.com/article/10.1007/BF01098186}.

\bibitem{PenkSkorn}
{\sc I.~B. Penkov and I.~A. Skornyakov}, {\em Projectivity and {$\D$}-affinity of flag supermanifolds}, Russ. Math. Surv., 40 (1985), pp.~233--234.
\newblock \url{https://iopscience.iop.org/article/10.1070/RM1985v040n01ABEH003546}.

\bibitem{PenkTikh}
{\sc I.~B. Penkov and A.~S. Tikhomirov}, {\em Linear ind-grassmannians}, Pure Appl. Math. Q., 10 (2014).
\newblock \url{https://math.constructor.university/penkov/papers/Ind-grassmanians.pdf}.

\bibitem{PolishchukGrassmannian}
{\sc A.~Polishchuk}, {\em {$\AA^{0|1}$}-torsors, quotients by free {$\AA^{0|1}$}-actions, and embeddings into {$\Pi$}-projective spaces and super-grassmannians {$G(1|1,n|n)$}}, Commun. Math. Phys., 402 (2023), pp.~2011--2029.
\newblock \url{https://par.nsf.gov/biblio/10521212}.

\end{thebibliography}

\end{document}